\documentclass[11pt,a4paper,reqno]{amsart}
\usepackage{fullpage, color, tikz}
\usepackage{amsthm, amsmath, amssymb, amsfonts, latexsym, hyperref}
\usepackage{shuffle}

\DeclareSymbolFont{Symbols}{OMS}{cmsy}{m}{n}
\DeclareMathSymbol{\Emptyset}{\mathord}{Symbols}{"3B}

\usepackage{MnSymbol}

\makeatletter
\@namedef{subjclassname@2020}{\textup{2020} Mathematics Subject Classification}
\makeatother

\newcommand{\red}[1]{\textcolor{red}{#1}}

\newcommand\Const{\mathsf{Const}}
\newcommand\const{\mathsf{const}}
\newcommand\Prec{\mathsf{prec}}
\newcommand\Resp{\mathsf{resp}}
\newcommand\Succ{\mathsf{succ}}
\newcommand\Subperms{\mathsf{SubPerms}}
\newcommand\Traces{\mathsf{Traces}}

\newcommand{\LE}{\mathcal{L}}

\newcommand{\im}{\mathrm{im}}
\newcommand{\PossIm}{\mathsf{PossIm}}
\newcommand{\MaxPossIm}{\mathsf{MaxPossIm}}
\newcommand{\Imax}{M}
\newcommand{\ior}{\protect{\overset{occ}{\lesssim} }}
\newcommand{\opr}{\overset{ord}{\boldsymbol{\rightarrow}}}
\newcommand{\oppo}{\overset{ord}{\preceq}}
\newcommand{\opponeq}{\overset{ord}{\prec}}

\newcommand{\centered}[1]{\begin{tabular}{l} #1 \end{tabular}}

\newcommand{\GETOUT}[1]{}
\newcommand{\CoreOcc}{\mathsf{CoreOcc}}

\theoremstyle{definition}
\newtheorem{theorem}{Theorem}[section]
\newtheorem{proposition}[theorem]{Proposition}
\newtheorem{lemma}[theorem]{Lemma}

\newtheorem{property}[theorem]{Property}

\newtheorem{definition}[theorem]{Definition}

\newtheorem*{remark*}{Remark}
\newtheorem{example}[theorem]{Example}
\newtheorem*{example*}{Example}
\newcommand{\equivcls}[1]{#1/{\sim}}
\def\Oh{\mathcal{O}}
\newcommand{\myimplies}{\Rightarrow}

\title{Characterizing traces of processes defined by precedence and response constraints: an order theory approach}
\author{Mark Dukes and Anton Sohn}
\address{UCD School of Mathematics and Statistics, University College Dublin, Dublin 4, Ireland.}
\email{mark.dukes@ucd.ie, anton.sohn@ucdconnect.ie}

\begin{document}
\begin{abstract}
In this paper we consider a general system of activities that can, but do not have to, occur. 
This system is governed by a set containing two types of constraints: precedence and response.
A precedence constraint dictates that an activity can only occur if it has been preceded by some other specified activity. 
Response constraints are similarly defined.
An execution of the system is a listing of activities in the order they occur and which satisfies all constraints.
These listings are known as {\it{traces}}.
Such systems naturally arise in areas of theoretical computer science and decision science.
An outcome of the freedom with which activities can occur is that there are many different possible executions, 
and gaining a combinatorial insight into these is a non-trivial problem.

We characterize all of the ways in which such a system can be executed. 
Our approach uses order theory to provide a classification in terms of the linear extensions of posets constructed from the constraint sets.
This characterization is essential in calculating the stakeholder utility metrics that have been developed by the first author that allow for quantitative comparisons of such systems/processes.
It also allows for a better understanding of the theoretical backbone to these processes. 
\end{abstract}

\begingroup
\def\uppercasenonmath#1{} 
\let\MakeUppercase\relax 
\maketitle
\endgroup

\section{Introduction}

The notion of an activity requiring a previous activity to have first occurred is an innate concept in discrete modelling. 
So too is the notion of an activity occurring as a response to some activity having occurred. 
Examples are plentiful and include the 
concurrency one finds in asynchronous systems~\cite{dienetal},
the precedence diagram method~\cite{pdm} in the area of scheduling/project planning, 
the precedence graph for testing conflict serializability in database management systems~\cite[Chpt. 17]{dbms},
and finite-state verification in software systems~\cite{fsv,aceto}.

Declarative processes~\cite{vda2009,pesic2007} have their origins in the area of business process modelling and were introduced as an alternative to the traditional imperative approaches~\cite{omg}.
The declarative paradigm consists of detailing, through a set of constraints on activities, how activities may occur with respect to one another. 
A declarative process can then happen in any possible way so long as it does not violate any of the specified constraints.

While declarative processes allow for extremely abstract and specialized constraints, the majority of those that appear in a modelling context are a small collection.
One such example is the {\it{precedence constraint}} that specifies if some activity $b$ were to occur, then another activity $a$ must have preceded it.
Another is the {\it{response constraint}} the specifies should an activity $a$ occur, then activity $b$ must occur as a response.

These natural constraints can appear in a wide variety of scenarios and are not restricted to the technically oriented examples mentioned above. 
Recently the first author introduced two utility metrics for declarative processes~\cite{dukes1,stakeholder}. 
These metrics were the proven solutions to a collection of stipulated axioms.
The purpose of this was to allow for quantitative comparisons of such processes against one another with respect to prescribed notions of stakeholder satisfaction~\cite{stakeholder}.
It also allowed for the comparison of declarative processes from the viewpoints of multiple stakeholders.
Examples of this comparative process analysis in action include a hospital emergency department admission processes~\cite{stakeholder} and aspects of the Federal Disaster Assistance Policy~\cite{fdap}.
When it comes to modelling human-designed processes, the most common constraints are precedence and response constraints. 

The ways in which a declarative process may be executed are called the {\it{traces}} of the process.
A necessary object in the calculation of the utility metrics is a listing of all traces~\cite{stakeholder}.
Determining this listing is a challenging enumeration problem since certain collections of constraints can create subtle trace inter-dependencies that are non-trivial to analyze.
Such a characterization is used to calculate the total number of traces along with the refined number of traces that satisfy stakeholder `satisfaction' constraints. 

In this paper we solve the problem of characterizing the traces of declarative processes in which the constraint set contains only precedence and response constraints.
These results are readily applicable to those areas of scheduling, asynchronous systems, database management systems, and finite-state verification that were previously mentioned.
Our approach is to represent the declarative process constraint set using partially ordered sets. 
We find that linear extensions of suitably defined posets are at the heart of the trace characterization problem.
The previous statement is a somewhat simplified account of our characterization but represents the essence of the solution.

In Section~\ref{sec:preliminaries} we define declarative processes and some of the necessary concepts. 
In Section~\ref{sec:general_case} we consider the most general case of constraint sets that contain precedence and response constraints. 
The derivation of the trace set is quite involved, but is achievable, and the main result of this section is Theorem~\ref{prec_resp_traces}.
In Section~\ref{sec:implementation} we describe how said result may be implemented in order to generate the trace set of a declarative process, making use of known algorithms in the literature.
In Section~\ref{sec:special_cases} we consider three special cases of the declarative processes studied in Section~\ref{sec:general_case}.
These special cases occur when constraints are either all precedence constrains, all response constraints, or all \textit{successor constraints} (to be defined at the start of Section~\ref{sec:preliminaries}). 
We are able to leverage properties of the relations induced by the constraint sets to give specialised expressions for the trace sets which are given in Theorems~\ref{prec_traces} and \ref{succ_traces}.

\section{Preliminaries}\label{sec:preliminaries}

A declarative process $D$ is a pair $(\Sigma,\Const)$ where $\Sigma$ is a set of $n$ activities and 
$\Const$ is a listing of constraints detailing existential and temporal dependencies between activities.
These processes were recently defined and studied in a series of papers ~\cite{fdap,stakeholder,dukes1}.
In this paper we will focus on two constraints: the precedence constraint $\Prec$, and the response constraint $\Resp$.
The constraint $\Resp(a,b)$ indicates that if activity $a$ occurs then activity $b$ must occur afterwards as a response to it.
Similarly $\Prec(a,b)$ indicates that if $b$ occurs then $a$ must have occurred before it.

We will also consider a third constraint $\Succ$, termed the successor constraint, that represents the logical conjunction of $\Prec$ and $\Resp$:
$$\Succ(a,b) ~\equiv ~ \Resp(a,b) \wedge \Prec(a,b).$$
Given activities $a$ and $b$, the constraint $\Succ(a,b)$ indicates that $a$ occurs if and only if $b$ occurs, and that $a$ occurs before $b$. Constraint sets consisting of successor constraints can be seen to be a special type of the more general precedence and response constraint sets.

A trace of a declarative process $D = (\Sigma, \Const)$ is a sequence of elements of $\Sigma$ that satisfies all constraints in $\Const$. 
We will represent sequences of elements of a set $A$ as words over the alphabet $A$.
For example, the sequence $(a,b,c,d)$ will be written as $abcd$.
Note that the empty sequence, denoted by $\epsilon$, may be a trace of $D$.
In the research so far on this topic, and also in this paper, we will be interested in \textit{first-passage traces}, which are those traces in which each activity $a \in \Sigma$ occurs at most once, so from here onward we use trace to mean first-passage trace.
We denote the set of traces of $D$ by $\Traces(D)$.
Throughout this paper we make the reasonable assumption that $\Sigma$ is always nonempty.
Moreover, we will assume that if a constraint between activities $a$ and $b$, such as $\Prec(a,b)$, is in $\Const$, then $a$ and $b$ are distinct elements of $\Sigma$.

\begin{example}\label{very:first:example}
Consider $D=(\Sigma,\Const)$ with $\Sigma=\{a,b,c\}$ and $\Const=\{\Resp(c,a),\Prec(b,a)\}$.
Then $$\Traces(D)=\{\epsilon, b, ba, cba, bca\}.$$
\end{example}

Given a set $\Sigma$, let $\Subperms(\Sigma)$ be the set of all permutations of all subsets of the set $\Sigma$. For example, 
\begin{align*}
\Subperms(\{a,b,c\}) = \{&\epsilon, a, b, c, ab, ba, ac, ca, bc, cb, abc, acb, bac, bca, cab, cba\}.
\end{align*}

One straightforward way of determining the set of all traces of a system is to check whether each sequence $\tau \in \Subperms(\Sigma)$ satisfies all of the constraints and to include it in $\Traces(D)$ if so.
A different approach would be to consider the $|\Const|$ different declarative systems $\{(\Sigma, {\const})\}_{\const \in \Const}$, 
generate the traces for each of these systems, and then take the intersection of the traces for the $|\Const|$ systems. While not as brute force as the first method, this is not far in the direction of an improvement.
Ideally, we would like a method that takes into account the specific structure of $\Const$ and can announce the traces of the system as a function of this structure. 
This is the motivation behind the direction we take in this paper.

Before beginning, there are several concepts best introduced now.
For $n$ a positive integer we denote by $[n]$ the set $\{ 1, \dots, n \}$.
Given a sequence $x = x_1 \dots x_k$ we define the image of $x$, denoted $\im(x)$, to be the set $\{x_1, \dots, x_k\}$ of terms in that sequence.
We define the restriction of a sequence $x$ to the set $A$, denoted $x|_A$, to be the subsequence of $x$ obtained by deleting from $x$ all those terms which are not in $A$. 
Given a binary relation $R$ and a set $A$, we define the restriction of $R$ to $A$, denoted by $R|_A$, to be the relation $R \cap (A \times A) = \{(a,b) \in R : a,b \in A \}$. We denote by $R^\oplus$ the reflexive and transitive closure of $R$ on the set on which it is defined, and we denote by $R^T$ the transpose of $R$.
If $\sim$ is an equivalence relation on the set $A$, then we denote by $[a]_\sim$ the $\sim$-equivalence class of $a \in A$ and we denote by $\equivcls{A}$ the set of all $\sim$-equivalence classes.

\subsection{Preorders and partial orders}

A \textit{preorder} on a set $P$ is a binary relation $\lesssim$ on $P$ that is reflexive and transitive. 
A \textit{partial order} on a set $P$ is a preorder $\preceq$ on $P$ which is also antisymmetric. 
We call the pair $(P, \preceq)$ a \textit{partially ordered set} or \textit{poset}.
Two elements $a,b \in P$ are said to be \textit{comparable} if either $a \preceq b$ or $b \preceq a$, and \textit{incomparable} otherwise.
If every pair of elements of $P$ is comparable, then $\preceq$ is a $\textit{total order}$, and we call the pair $(P, \preceq)$ a totally ordered set. 
A subset $A$ of $P$ is called an \textit{anti-chain} of the poset $(P, \preceq)$ if no two elements of $A$ are comparable, and a \textit{chain} if any two elements of $A$ are comparable.
A chain is called \textit{maximal} if it is not contained in a larger chain.
A chain $A$ is called \textit{saturated} if there does not exist $c \in P \backslash A$  such that $a \prec c \prec b$ for some $a,b \in A$ and such that $A \cup \{c\}$ is a chain.
A subset $I$ of $P$ is called an \textit{down-set} of the preordered set $(P, \lesssim)$ if for all $a \in P$ and $b \in I$ it satisfies $a \lesssim b \Rightarrow a \in I$.
The \textit{principal down-set} generated by $b \in P$ is the set $\{ a \in P : a \lesssim b \}$.
The set of all down-sets of a $(P, \lesssim)$ when ordered by inclusion forms a poset denoted by $\mathcal{J}((P, \lesssim))$.
An \textit{induced subposet} of a poset $(P, \preceq_P)$ is a poset $(Q, \preceq_Q)$ such that $Q \subseteq P$ and for all $a, b \in Q, \, a \preceq_Q b \iff a \preceq_P b$ (or equivalently $\preceq_Q \, = \, \preceq_P \! |_Q$).

The \textit{Hasse diagram} of the preorder $\lesssim$ is a transitive reduction of the directed graph $(P, \lesssim)$. 
The Hasse diagram of a preorder need not be unique, since we can permute the vertices in any cycle to obtain an equally valid Hasse diagram.
A preorder is antisymmetric (and hence a partial order) exactly when its Hasse diagram is acyclic, and in this case the Hasse diagram is unique.

\subsection{Linear extensions}

A \textit{linear extension} of a poset $(P,\preceq)$ is a total ordering of $P$ which respects $\preceq$.
That is, a totally ordered set $\sigma = (P, \preceq_\sigma)$ such that for all $a,b \in P$, $a \preceq b \myimplies a \preceq_\sigma b$.
If $\sigma = (\{a_1, \dots, a_n\}, \preceq_\sigma)$ is the totally ordered set such that $a_1 \prec_\sigma \dots \prec_\sigma a_n$, then by abuse notation we also refer to it as the sequence $\sigma = a_1 \dots a_n$.

\begin{example}
    For example, let $P$ be the poset on $\{a,b,c,d\}$ with Hasse diagram:
    \begin{center}
    \begin{tikzpicture}
        \node (a) at (0,0) {$a$};
        \node [above of=a] (c)  {$c$};
        \node [right of=a ] (b) {$b$};
        \node [above of=b] (d) {$d$};
        \draw [->] (a) -- (c);
        \draw [->] (b) -- (c);
        \draw [->] (b) -- (d);
    \end{tikzpicture}
    \end{center}
    Then $\sigma = abcd$ is a linear extension of $P$, whereas $\sigma= cbad$ is not.
    The set of all linear extensions of $P$ is 
    $$\{abcd, bacd, abdc, badc, bdac\}.$$
\end{example}

\subsection{Independence systems}
An \textit{independence system} is a pair $(E, \mathcal{I})$ consisting of a finite set $E$, and a non-empty collection $\mathcal{I}$ of 
subsets of $E$ which satisfies the condition $I \in \mathcal I$ and $I' \subseteq I \myimplies I' \in \mathcal{I}$.
We call the elements of $\mathcal{I}$ \textit{independent sets}.

\section{Precedence-and-response-only constraint sets}\label{sec:general_case}

In this section we investigate the most general case of this paper.
Consider a declarative process $D = (\Sigma, \Const)$ where $\Const$ consists only of precedence and response constraints. 
We refer to such a declarative process as a \textit{precedence-and-response-only} declarative process.
Our main result is a description for $\Traces(D)$ in terms of order theoretic concepts related to the interdependencies stated in the constraint set.
The following is an example of such a process and will be used to illustrate our approach throughout the section.

\begin{example}\label{prec:resp:example}
    The declarative process $D = (\Sigma,\Const)$ where $\Sigma=\{a,b,c,d,e\}$ and 
	$$\Const=\{ \Resp(b,a), \Resp(c,a), \Resp(d,e), \Resp(e,c), \Prec(a,d), \Prec(b,d), \Prec(d,e) \}$$ is a precedence-and-response-only declarative process.
\end{example}

Our approach to determining $\Traces(D)$ will involve first determining the set
$$\PossIm(D) := \{ \im(\tau) : \tau \in \Traces(D)\},$$
the collection of all possible sets that occur as the image of some trace of $D$, and then arranging each of these sets in all possible ways that the constraint set allows.
Notice that the way in which $\Resp(b,a)$ constrains the activities $a$ and $b$ can be encapsulated in two independent statements:
\begin{enumerate}
    \item if $b$ occurs then $a$ must also occur, and
    \item if $a$ and $b$ both occur then $b$ must appear before $a$.
\end{enumerate}
The same description can be seen to hold for precedence constraints, the only difference being that in the first statement, the activities would be switched.

The statement (i) tells us something about the possible images that a trace of the declarative process can have.
For example, we know that there cannot be a trace of $D$ with image $\{b,c\}$, since due to the constraint $\Resp(b,a) \in \Const$, the occurrence of $b$ \textit{implies the occurrence of} $a$. 
Statements of this kind are key to the first step of our approach, determining the set $\PossIm(D)$.

Statement (ii) provides information about the order in which activities may occur in a trace of the declarative process.
For example, we know that $acb$ cannot be a trace of $D$, since due to the constraint $\Resp(b,a) \in \Const$, the \textit{order $ba$ must be preserved}.
Statements of this kind are key to the second step of our approach, arranging sets in all possible ways to form valid traces. 
We will need to introduce two relations which correspond to statements of type (i) and (ii) respectively.

\begin{definition}\label{ior}
    Let $D = (\Sigma, \Const)$ be a precedence-and-response-only declarative process.
    Then the \textit{implied-occurrence relation} of $D$ is the relation 
    $$\ior ~:=~ \{ (a,b) \in \Sigma^2 : \Prec(a,b) \in \Const \text{ or } \Resp(b,a) \in \Const \}^\oplus$$
    on $\Sigma$.
\end{definition}

Notice that the relation $\ior$ is a preorder on $\Sigma$ by construction.
From the previous discussion, the implied-occurrence relation of $D$ can be easily seen to satisfy the following property. 

\begin{property}\label{ior_property}
    If $a \ior b$ for some $a,b \in \Sigma$, then in any trace in which $b$ occurs, $a$ must also occur. 
\end{property}

\begin{definition}\label{opr}
    Let $D = (\Sigma, \Const)$ be a precedence-and-response-only declarative process. 
    Then the \textit{order-preserving relation} of $D$ is the relation 
    $$\opr ~ := ~ \{ (a,b) \in \Sigma^2 : \Prec(a,b) \in \Const \text{ or } \Resp(a,b) \in \Const \}$$
    on $\Sigma$.
\end{definition}

From the previous discussion, the order-preserving relation of $D$ can be easily seen to satisfy the following property. 

\begin{property}\label{opr_property}
    If $a \opr b$ for some $a,b \in \Sigma$, then $a$ must appear before $b$ in any trace in which both activities occur.
\end{property}

Note that $\opr$ is not a preorder. 
The fact that we do not take the transitive closure in the definition of the order preserving relation is a very important subtlety which will be explained later.

\begin{example}
    Consider again the declarative process $D$ from Example~\ref{prec:resp:example}.
    The Hasse diagram of the preordered set $(\Sigma, \ior)$ is given in Figure~\ref{fig_pre}.
    \begin{figure}
        \begin{center}
            \begin{tikzpicture}
                \node (one) at (0,0) {$a$};
                \node (two) at (-1,1) {$b$};
                \node (three) at (1,1) {$c$};
                \node (four) at (0,2) {$d$};
                \node (five) at (1.5,2) {$e$};
                \draw[->] (one) -- (two);
                \draw[->] (one) -- (three);
                \draw[->] (two) -- (four);
                \draw[->] (three) -- (four);
                \draw[->] (four) to [out=30,in=150] (five);
                \draw[->] (five) to [out=210,in=330] (four); 
            \end{tikzpicture}
        \end{center}
        \caption{
		Hasse diagram for the preordered set 
		from 
		Example
		\protect{\ref{prec:resp:example}}  
		}
		\label{fig_pre}
    \end{figure}
    The down-sets of $(\Sigma, \ior)$ are $\Emptyset, \{a\}, \{a,b\},$ $\{a,c\}, \{a,b,c\}$, and $\Sigma$.
    We also have $\opr ~=~ \{ (b,a), (c,a), (d,e), (e,c), (a,d), (b,d) \}$.
\end{example}

We now move towards classifying $\PossIm(D)$.

\begin{lemma}\label{ior_lemma}
    If $\tau \in \Traces(D)$ then $\im(\tau)$ is a down-set of $(\Sigma, \ior)$.
\end{lemma}

\begin{proof}
    Let $a,b \in \Sigma$ be such that $b \in \im(\tau)$ and $a \ior b$. 
    Then $a \in \im(\tau)$ since $\ior$ satisfies Property~\ref{ior_property}.
\end{proof}

The converse of the above lemma does not hold, i.e. a down-set of $(\Sigma, \ior)$ need not be the image of some trace.
Consider for example $D = ( \{a,b\}, \{ \Prec(a,b), \Prec(b,a) \})$. Here $\{a,b\}$ is certainly a down-set of $(\Sigma, \ior)$, but neither the sequence $ab$ nor $ba$ satisfies both constraints. 
The following construction will allow us to characterize the down-sets which \textit{are} the image of some trace.
It will then also be used to determine $\Traces(D)$.

\begin{definition}\label{oppo}
    Let $D = (\Sigma, \Const)$ be a precedence-and-response-only declarative process.
    Recall $\ior$ and $\opr$ from Definitions~\ref{ior}~and~\ref{opr}, respectively.
    If $I$ is a down-set of $(\Sigma, \ior)$ then define $\oppo_I$ to be the relation $(\opr|_I)^\oplus$ on $I$.
\end{definition}

Naturally, we write $a \opponeq_I b$ when $a \oppo_I b$ and $a \neq b$.
Notice that $\oppo_I$ is a preorder on $I$ by construction, and hence it is a partial order if and only if it is antisymmetric. 
 
\begin{lemma}\label{opr_lemma}
    Let $\tau \in \Traces(D)$ and let $I = \im(\tau)$. 
    If $a_0,a_m \in I$ are such that $a_0 \opponeq_I a_m$, then $a_0$ appears before $a_m$ in $\tau$.
\end{lemma}

\begin{proof}
    By definition of $\oppo_I$ we have that $a_0 \opr a_1 \opr \dots \opr a_m$ for some $a_1, \dots, a_{m-1} \in I$.
    Since $\opr$ satisfies Property~\ref{opr_property}, we have that $a_{i-1}$ appears before $a_i$ in $\tau$ for all $i \in [m]$.
    Hence $a_0$ must appear before $a_m$ in $\tau$.
\end{proof}

So the traces of $D$ must comply with $\oppo_I$ by construction. 
In fact, we will see later that $\oppo_I$ completely describes the ways in which we can and cannot arrange the elements of $I$ to form a trace.

It is important to mention that in the definition of $\oppo_I$ we restrict to $I$ \textit{before} taking the transitive closure.
The following example illustrates why this is necessary.

\begin{example}
    Consider the declarative process $D = (\{a,b,c\}, \{ \Prec(a,b), \Resp(b,c) \})$. 
    We have that 
	$
	\ior \, = \{ (a,a), (b,b), (c,c), (a,b), (c,b) \}$ and $
	\opr \, = \{ (a,b), (b,c) \}$. 
    Recall that $\opr$ is constructed to have the property that if $a \opr b$ for some $a,b \in \Sigma$, then $a$ must appear before $b$ in any trace in which both activities occur.
    It would seem reasonable to assume that $\opr\!\vphantom{=}^\oplus$ would still respect this property, since if $a$ appears before $b$ which in turn appears before $c$, 
	then $a$ would have to appear before $c$, but there is an error in this reasoning.
    In our example, $\opr\!\vphantom{=}^\oplus$ contains $(a,c)$, and $I = \{a, c\}$ is a down-set of $(\Sigma, \ior)$ and so is a contender for being the image of a trace.
    If we restrict $\opr\!\vphantom{=}^\oplus$ to $I$, it still contains $(a,c)$, which means that the sequence $ca$ cannot be in $\Traces(D)$.
    Again this seems reasonable, but by examining the constraint set, note that if $b$ does not occur in a trace then there is no restriction on the order in which $a$ and $c$ must occur.	
    This shows that $ca$ is, in fact, in $\Traces(D)$.
\end{example}

Before classifying $\PossIm(D)$, and subsequently $\Traces(D)$, we prove the following lemma.

\begin{lemma}\label{LE_is_trace}
    Let $I$ be a down-set of $(\Sigma, \ior)$ such that $\oppo_I$ is antisymmetric. 
    Then $\LE((I, \oppo_I)) \subseteq \Traces(D)$.
\end{lemma}

\begin{proof}
    First note that $(I, \oppo_I)$ is a poset since $\oppo_I$ is already reflexive and transitive by construction.
    Let $\tau \in \LE((I, \oppo_I))$.
    Then $\im(\tau) = I$.
    If $I = \Emptyset$ then $\tau = \epsilon \in \Traces(D)$ vacuously, so suppose that $I \neq \Emptyset$.
    We wish to show that $\tau$ satisfies some arbitrary $\Prec(a,b), \Resp(c,d) \in \Const$.
    
    Suppose that $b \in \im(\tau)$ since otherwise $\tau $ satisfies $ \Prec(a,b)$ vacuously.
    From the definition of $\ior$, we have that $a \ior b$ which implies that $a \in I = \im(\tau)$ since $I$ is a down-set of $(\Sigma, \ior)$.
    From the definition of $\opr$, we have that $a \opr b$ so $a \opr|_I b$ and finally $a \opponeq_I b$.
    Hence $a \prec_\tau b$ since $\tau$ is a linear extension of $(I, \oppo_I)$.
    Thus $a$ appears before $b$ in $\tau$, so $\tau $ satisfies $ \Prec(a,b)$.
    By exactly the same reasoning we can show that $\tau $ satisfies $ \Resp(c,d)$.
    Therefore $\tau \in \Traces(D)$.
\end{proof}

We may now state and prove the classification of $\PossIm(D)$.

\begin{proposition}\label{prec_resp_possim}
    If $D$ is a precedence-and-response-only declarative process, then $\PossIm(D)$ is the set of all down-sets $I$ of $(\Sigma, \ior)$ for which $\oppo_I$ is antisymmetric.
\end{proposition}

\begin{proof}
    We prove that $I \in \PossIm(D)$ if and only if $I$ is a down-set of $(\Sigma, \ior)$ and $\oppo_I$ is antisymmetric.
	Let $I \in \PossIm(D)$. 
    Then $I = \im(\tau)$ for some $\tau \in \Traces(D)$. 
    Hence $I$ is a down-set of $(\Sigma, \ior)$ by Lemma~\ref{ior_lemma}.
    Suppose for a contradiction that $\oppo_I$ is not antisymmetric, and let $a,b \in I$ be such that $a \opponeq_I b$ and $b \opponeq_I a$.
    By Lemma~\ref{opr_lemma}, $a$ must appear before $b$ which in turn must appear before $a$ in $\tau$.
    So $a$ appears more than once in $\tau$, a contradiction.
      
    Conversely, let $I$ be a down-set of $(\Sigma, \ior)$ such that $\oppo_I$ is antisymmetric. 
    Since $\oppo_I$ is antisymmetric the pair $(I, \oppo_I)$ is a poset.
    Let $\tau \in \LE((I, \oppo_I))$.
    By Lemma~\ref{LE_is_trace} we have $\tau \in \Traces(D)$.
    But $\im(\tau) = I$ so $I \in \PossIm(D)$.
\end{proof}

\newcommand\tpone{
	\begin{tikzpicture}[scale=0.5]
		\tiny
    	\node [fill,circle,scale=0.3] (three) at (0,0) [label={right:$a$}]{};
	\end{tikzpicture}
}
\newcommand\tptwo{
	\begin{tikzpicture}[scale=0.5]
		\tiny
    	\node [fill,circle,scale=0.3] (two) at (0,0) [label={right:$b$}]{};
    	\node [fill,circle,scale=0.3] (one) at (0,0.75) [label={right:$a$}]{};
		\draw[->] (two) -- (one);
	\end{tikzpicture}
}
\newcommand\tpthree{
	\begin{tikzpicture}[scale=0.5]
		\tiny
    	\node [fill,circle,scale=0.3] (three) at (0,0) [label={right:$c$}]{};
    	\node [fill,circle,scale=0.3] (one) at (0,0.75) [label={right:$a$}]{};
		\draw[->] (three) -- (one);
	\end{tikzpicture}
}
\newcommand\tpfour{
	\begin{tikzpicture}[scale=0.5]
		\tiny
    	\node [fill,circle,scale=0.3] (two) at (0,0) [label={left:$b$}]{};
    	\node [fill,circle,scale=0.3] (one) at (0.375,0.75) [label={right:$a$}]{};
    	\node [fill,circle,scale=0.3] (three) at (0.75,0) [label={right:$c$}]{};
		\draw[->] (two) -- (one);
		\draw[->] (three) -- (one);
	\end{tikzpicture}
}
\newcommand\tpfive{
	\begin{tikzpicture}[scale=0.5]
		\tiny
    	\node [fill,circle,scale=0.3] (two) at (0,0) [label={left:$b$}]{};
    	\node [fill,circle,scale=0.3] (one) at (0,0.8) [label={right:$a$}]{};
        \node [fill,circle,scale=0.3] (four) at (-0.4,1.4) [label={left:$d$}]{};
        \node [fill,circle,scale=0.3] (three) at (0.4,1.4) [label={right:$c$}]{};
    	\node [fill,circle,scale=0.3] (five) at (0,2) [label={right:$e$}]{};
		\draw[->] (two) -- (one);
        \draw[->] (one) -- (four);
        \draw[->] (four) -- (five);
        \draw[->] (five) -- (three);
        \draw[->] (three) -- (one);
	\end{tikzpicture}
}

\begin{example}\label{prec_resp_table}
    Consider again the declarative process $D$ from Example~\ref{prec:resp:example}.
    The relations $\opponeq_I$ corresponding to each of the down-sets $I$ of $(\Sigma, \ior)$, as well as their Hasse diagrams, are given in Table~\ref{hdias}.
	\begin{table}
    \begin{center}
        \begin{tabular}{|c|c@{\quad}|@{\quad}c@{\quad}|@{\quad}c|} \hline
            $k$ & $I_k$ & $\opponeq_{I_k}$ & Hasse diagram \\ \hline \hline
            \centered{$1$} & \centered{$\Emptyset$} & \centered{$\Emptyset$} &    \\ \hline
            \centered{$2$} & \centered{$\{a\}$} & \centered{$\Emptyset$} & \centered{\tpone} \\ \hline
            \centered{$3$} & \centered{$\{a,b\}$} & \centered{$\{(b,a)\}$} & \centered{\tptwo} \\ \hline
            \centered{$4$} & \centered{$\{a,c\}$} & \centered{$\{(c,a)\}$} & \centered{\tpthree} \\ \hline
            \centered{$5$} & \centered{$\{a,b,c\}$} & \centered{$\{(b,a),(c,a)\}$} & \centered{\tpfour} \\ \hline
            \centered{$6$} & \centered{$\Sigma$} & \centered{$\{(b,a),(c,a),(d,e),(e,c),(a,d),(b,d),$$(a,e),(a,c),$\\$(d,c),(d,a),(e,a),(e,d),(c,d),(c,e),(b,e),(b,c)\}$} & \centered{\tpfive} \\ \hline
        \end{tabular}
    \end{center}
	\caption{The down-sets of Example~\ref{prec_resp_table} along with the relations detailing the order in which they must occur.\label{hdias}}
	\end{table}
    From the Hasse diagrams in Table~\ref{hdias} we can see that only $\oppo_\Sigma$ is not antisymmetric, 
	and so we use Proposition~\ref{prec_resp_possim} to find that 
	$$\PossIm(D) = \{ I_1, I_2, I_3, I_4, I_5 \} = \{ \Emptyset, \{a\}, \{a,b\}, \{a,c\}, \{a,b,c\} \}.$$
\end{example}

Now if $I \in \PossIm(D)$, then it is the image of at least one trace of $D$, and the order in which activities can occur in these traces is bounded by $\oppo_I$. 
This means that arranging $I$ in all possible ways that the constraint set allows turns out to be equivalent to generating the set $\LE((I, \oppo_I))$.
This leads to the following classification of $\Traces(D)$.

\begin{theorem}\label{prec_resp_traces}
    Recall $\ior$ and $\oppo_I$ from Definitions~\ref{ior} and~\ref{oppo}, respectively.
    If $D$ is a precedence-and-response-only declarative process, then
    $$\Traces(D) = \displaystyle\bigcup_{I \in \PossIm(D)} \LE((I, \oppo_I)),$$
    where $\PossIm(D)$ is the set of all down-sets $I$ of $(\Sigma, \ior)$ such that $\oppo_I$ is antisymmetric.
\end{theorem}

\begin{proof}
    The fact that $\PossIm(D)$ is the set of all down-sets $I$ of $(\Sigma, \ior)$ such that $\oppo_I$ is antisymmetric is proven in Proposition~\ref{prec_resp_possim}.
	Now let $\tau \in \Traces(D)$. 
    By Lemma~\ref{ior_lemma}, $\im(\tau) = I$ for some down-set $I$ of $(\Sigma, \ior)$.
    We have that $I \in \PossIm(D)$ so $\oppo_I$ is antisymmetric and hence $(I, \oppo_I)$ is a poset.
	Now let $a,b \in I$ such that $a \opponeq_I b$. 
    By Lemma~\ref{opr_lemma}, $a$ must appear before $b$ in $\tau$, that is, $a \prec_\tau b$.
    Hence $\tau \in \LE((I, \oppo_I))$ and so 
	$$\Traces(D) \subseteq \displaystyle\bigcup_{I \in \PossIm(D)} \LE((I, \oppo_I)).$$
	Conversely, let $I \in \PossIm(D)$.
    Then $\oppo_I$ is antisymmetric, so $\LE((I, \oppo_I)) \subseteq \Traces(D)$ by Lemma~\ref{LE_is_trace}.
	Therefore 
	$$\Traces(D) \supseteq \displaystyle\bigcup_{I \in \PossIm(D)} \LE((I, \oppo_I)).\qedhere$$
\end{proof}

\begin{example}
    Applying Theorem~\ref{prec_resp_traces} to the declarative process $D$ from Example~\ref{prec:resp:example},
	we find that 
    $$\Traces(D) = \displaystyle\bigcup_{k=1}^5 \LE((I_k, \preceq_{I_k})) = \{ \epsilon, a, ba, ca, bca, cba \}.$$
	Note the the possible images were derived in Example~\ref{prec_resp_table}.
\end{example}

\section{Implementation discussion}\label{sec:implementation}

In this section, we discuss some aspects of implementing Theorem~\ref{prec_resp_traces}.
To use Theorem~\ref{prec_resp_traces} to calculate $\Traces(D)$ we will first need to generate $\PossIm(D)$, 
and then for each $I \in \PossIm(D)$ generate the linear extensions of each $(I, \oppo_I)$.
We will split up the different challenges of the implementation into subsections.

\subsection{Preorders and posets}
The preorder $\ior$ on $\Sigma$ induces a partial order as follows.
The relation $\sim$ on $\Sigma$ defined by $a \sim b$ if $a \ior b$ and $b \ior a$ can be seen to be an equivalence relation.
Define the relation $\leq$ on $\equivcls{\Sigma}$ by $[a]_\sim \leq [b]_\sim$ if $a \ior b$.
A well-established fact is that $\leq$ is both well defined and a partial order on the set $\equivcls{\Sigma}$.
These, and other properties of the equivalence relation $\sim$, can be found in Schr\"oder~\cite[Prop. 5.2.4]{Schroder} and also Stanley~\cite[Exercises 3.1 and 3.2]{stanley}.
Moreover, the posets $\mathcal{J}((\equivcls{\Sigma}, \leq))$ and $\mathcal{J}((\Sigma, \ior))$ of down-sets ordered by inclusion are isomorphic under the map 
\begin{align*}
\psi & : \mathcal{J}((\equivcls{\Sigma}, \leq)) \mapsto \mathcal{J}((\Sigma, \ior))\\
\psi & (I') := \bigcup_{S \in I'} S.
\end{align*}

Since the only aspect of $(\Sigma, \ior)$ in which we are interested is its down-sets, we may now restrict our attention to the case where $(\Sigma, \ior)$ is a poset. 
If $(\Sigma, \ior)$ is not a poset, then we can just work with the poset $(\equivcls{\Sigma}, \leq)$ whose down-sets are in 1-1 correspondence with those of $(\Sigma, \ior)$. 
Consequently, for the remainder of this section we assume $(\Sigma, \ior)$ to be a poset.

\subsection{Generating $\PossIm(D)$}
To determine $\PossIm(D)$ one could naively generate every single down-set of $(\Sigma, \ior)$ and then check, by using Proposition~\ref{prec_resp_possim}, which of the down-sets belong to $\PossIm(D)$.
However, the set $\PossIm(D)$ turns out to have a property that allows us to avoid the unnecessary generation of many down-sets which would be ultimately discarded by this method.
For a down-set $I$ of a $(\Sigma, \ior)$, define 
\begin{align*}
\mathsf{max}(I) &:= \{ x \in \Sigma ~:~ x \ior y \myimplies x = y \mbox{ for all }y \in \Sigma\}.
\end{align*}
It is well known~\cite[p.282]{stanley} that $\mathsf{max}$ is a bijection from the down-sets to the anti-chains of a poset, 
and also that its inverse is the downward closure operation 
\begin{align*}
\downarrow \! A &:= \{ x \in \Sigma ~:~ x \ior y \text{ for some } y \in A \}.
\end{align*}
Therefore generating $\PossIm(D)$ is tantamount to generating the set 
\begin{align*}
\MaxPossIm(D) &:= \{ \mathsf{max}(I) ~:~ I \in \PossIm(D) \}.
\end{align*}

\begin{lemma}\label{maxpossim}
    Let $D$ be a precedence-and-response-only declarative process.
    Then the pair $(\Sigma, \MaxPossIm(D))$ is an independence system.
\end{lemma}

\begin{proof}
    Let $A \in \MaxPossIm(D)$. Then $A = \mathsf{max}(I)$ for some $I \in \PossIm(D)$.
    Let $B \subseteq A$ and define $J = \, \downarrow \! B$. 
    Clearly $J \subseteq I$ and $J$ is a down-set of $(\Sigma, \ior)$. 
    Suppose for a contradiction that $J \notin \PossIm(D)$.
    By Proposition~\ref{prec_resp_possim}, we must have that $\preceq_J$ is not antisymmetric.
    However, 
	\begin{align*}
	(\preceq_J) &= (\opr|_J)^\oplus \subseteq (\opr|_I)^\oplus = (\oppo_I),
	\end{align*} 
	so $\oppo_I$ is also not antisymmetric, contradicting $I \in \PossIm(D)$.
    Hence $J \in \PossIm(D)$ and so $B \in \MaxPossIm(D)$.
\end{proof}

If we have an algorithm which indicates whether or not a given set is independent, 
then the elements of an independence system can be enumerated efficiently using a lexicographical tree approach, as is illustrated in the following example.

\begin{example}
    Consider the collection $\mathcal{I}$ of all subsets of $[4]$ whose sum does not exceed $5$. 
    It is not hard to see that $([4], \mathcal{I})$ is an independence system.
    Given a subset of $[4]$, we can straightforwardly check whether it is an independent set.
    Let us enumerate all independent sets using a lexicographical tree.
    \begin{center}
        \begin{tikzpicture}
        [
            level 1/.style={sibling distance=25mm},
            level 2/.style={sibling distance=10mm},
            level 3/.style={sibling distance=7mm},
        ]
        	\node {Root}
        		child {
                        node {1}
                        child { 
                            node {2}
                            child {node {\red{3}}}
                            child {node {\red{4}}}  
                        }
                        child {
                            node {3}
                            child {node {\red{4}}}
                        }
                        child {node {4}}
                    }
        		child {
        		    node {2}
        		    child {
                            node {3}
                            child {node {\red{4}}}
                        }
                        child {node {\red{4}}}
        		}
                    child {
        		    node {3}
        		    child {node {\red{4}}}
                    }
                    child {
                        node {4}
                    };
        \end{tikzpicture}
    \end{center}
    Each independent set corresponds to a (non-red) node of the tree by considering the path from the root to that node. 
    If a node corresponds to a set which we determine not to be independent, then we delete it (coloured red above).
    Otherwise, we create the children of the node, which are just those elements of $[4]$ which succeed it in the total ordering.
\end{example}

To enumerate the elements of $\MaxPossIm(D)$ (i.e. the independent sets) using a lexicographical tree, the underlying set $\Sigma$ must be totally ordered,
and so we simply assign it an arbitrary total ordering.
Let $A$ be the set of nodes on the path from the root to a node $x$ of the tree. 
The following algorithm indicates whether or not $A$ is independent, given that $A \backslash \{x\}$ is independent.
\begin{enumerate}
    \item[IC1] Ensure that $x$ is incomparable with all other elements of $A$. 
    If it is not, then $A$ is not an anti-chain of $(\Sigma, \ior)$ and can therefore not be independent, so we are done.
    \item[IC2] Otherwise, take the downward closure of $A$ to obtain a down-set $I = \,\downarrow \! A$ of $(\Sigma, \ior)$.
    \item[IC3] Restrict $\opr$ to $I$, and take the reflexive and transitive closure to obtain $\oppo_I = (\opr|_I)^\oplus$.
    \item[IC4] Check whether $\oppo_I$ is antisymmetric. If it is, then $A$ is independent and $I \in \PossIm(D)$.
\end{enumerate}

The most expensive step of this algorithm is IC3, taking the transitive closure of a relation on the set $|I|$.
This can be accomplished using the Floyd-Warshall algorithm~\cite{floyd,roy,warshall}, which has time complexity $\Oh(|I|^3)$. 
Consequently the entire algorithm has time complexity $\Oh(n^3)$ where $n = |\Sigma|$.
Given that one can determine in constant time whether a set is independent, 
one can enumerate the independent sets of an independence system $(E, \mathcal{I})$ in $\Oh(|E|\cdot |\mathcal{I}|)$ time using a lexicographical tree approach~\cite{lawler_etal}.
As a result of this, the time complexity for enumerating $\PossIm(D)$ is $\Oh(n^4|\PossIm(D)|)$.

\subsection{Generating the linear extensions}
For every down-set $I \in \PossIm(D)$ we can now use Pruesse and Ruskey's algorithm~\cite{pruesse_and_ruskey} to generate the set of linear extensions $\LE((I, \oppo_I))$.
Note that in the process of generating $\PossIm(D)$, we have already found every $\oppo_I$ that we need.
The Pruesse and Ruskey algorithm generates the set $\LE(P)$ with time complexity $\mathcal{O}(e(P))$ where $e(P)$ is the number of linear extensions of a poset $P$.

\subsection{Further remarks}
The above discussion outlines how standard algorithms in the literature, for which the complexity is well-known, can be used to generate the set $\Traces(D)$ from the expression given in Theorem~\ref{prec_resp_traces}. 
We would like to note, however, that an optimal trace generating algorithm is not a main goal of this paper. 
Instead, our focus is presenting the set of traces in a form in which each trace appears in a unique way and allows us to appreciate the discrete structure underpinning the trace sets.
Theorem~\ref{prec_resp_traces} shows us that every trace can be written as a unique pair $(I,\pi)$ where $I$ is a down-set of the relation $\ior$ and $\pi$ is a linear extension of the poset $(I,\oppo_I)$.

A motivating factor for doing this is the calculation of stakeholder utility metrics for declarative processes that appeared in the papers \cite{fdap,stakeholder}.
These stakeholder utility metrics require one to calculate the number of traces in a declarative process, and then determine the number of such traces that satisfy additional stakeholder constraints -- one for each such stakeholder.
In order to analyse such processes and make judgments about what sets of stakeholder constraints might `maximize' such a utility measure, being able to specify those activity sets that produce the largest set of linear extensions is of benefit.

\section{Special Cases}\label{sec:special_cases}

In this section we will consider declarative processes that consist only of a single type of constraint. 
We will discuss said properties and what makes them interesting, and use them to give specialized classifications, as well as remarks on their use.

\subsection{Precedence-only constraint sets}

Consider a declarative process $D = (\Sigma, \Const)$ where $\Const$ consists only of precedence constraints. 
We refer to such a declarative process as a \textit{precedence-only} declarative process.
In this case, it follows from Definition~\ref{ior} that $\ior$ is the relation $\{(a,b) \in \Sigma^2 : \Prec(a,b) \in \Const \}^\oplus$ on $\Sigma$ and from Definition~\ref{opr} that $\opr$ is the relation $\{(a,b) \in \Sigma^2 : \Prec(a,b) \in \Const \}$ on $\Sigma$.
Notice that $\ior$ is the same as $(\opr)^\oplus$, an important fact that we will make use of.

\begin{example}\label{prec_example}
    Consider the declarative process $D = (\Sigma, \Const)$ where $\Sigma= \{a,b,c,d,e,f\}$ and $\Const=\{ \Prec(a,c), \Prec(b,c), \Prec(c,d), \Prec(d,e), \Prec(e,d), \Prec(d,f) \}$. 
    The Hasse diagram of the corresponding preordered set $(\Sigma, \ior)$ is given in Figure~\ref{fig:prec_hasse}.
    
    \begin{figure}[ht]
        \centering
        \begin{tikzpicture}
            \node (f) at (0,3) {$f$};
            \node (e) at (1,2) {$e$};
            \node (d) at (0,2) {$d$};
            \node (c) at (0,1) {$c$};
            \node (b) at (0.7,0) {$b$};
            \node (a) at (-0.7,0) {$a$};
            \draw[->] (a) -- (c);
            \draw[->] (b) -- (c);
            \draw[->] (c) -- (d);
            \draw[->] (d) to [out=30,in=150] (e);
            \draw[->] (e) to [out=210,in=330] (d);
            \draw[->] (d) -- (f);
        \end{tikzpicture}
        \caption{Hasse diagram of the preordered set from Example~\protect{\ref{prec_example}}}
        \label{fig:prec_hasse}
    \end{figure}
\end{example}

As before, for each down-set $I$ of $(\Sigma, \ior)$ we define $\oppo_I$ to be the relation $(\opr|_I)^\oplus$ on $I$. 
The following lemma demonstrates that, unlike in the precedence-and-response-only case, we may equivalently take the transitive closure before restricting to $I$ in the definition of $\oppo_I$. 

\begin{lemma}\label{prec_lemma}
    Let $D = (\Sigma, \Const)$ be a precedence-only declarative process.
    If $I$ is a down-set of $(\Sigma, \ior)$, then $\oppo_I \,\, := (\opr|_I)^\oplus = (\opr\!\vphantom{=}^\oplus)|_I$.
    Furthermore, $\oppo_I \,\, = \ior|_I$.
\end{lemma}

\begin{proof}
    Suppose $(a,b) \in (\opr|_I)^\oplus$. We note that $a,b \in I$.
    Since $(a,b)$ is in the transitive closure there must exist $a_1, \dots a_m \in I$ such that $a (\opr\!|_I) a_1 (\opr\!|_I) \dots (\opr\!|_I) a_m (\opr\!|_I) b$ which implies $a \opr a_1 \opr \dots \opr a_m \opr b$.
    This means that $a (\opr\!\vphantom{=}^\oplus) b$ and so $a (\opr\!\vphantom{=}^\oplus)|_I b$.

	Conversely, suppose $(a,b) \in (\opr\!\vphantom{=}^\oplus)|_I$.
    Again, we note that this means $a,b \in I$.
    As $ a (\opr\!\vphantom{=}^\oplus) b$ there must exist $a_1, \dots a_m \in \Sigma$ such that $a \opr a_1 \opr \dots \opr a_m \opr b$.
    This implies $a_i (\opr\!\vphantom{=}^\oplus) b$ for all $ i \in [m]$, which in turn implies $a_i \ior b$  for all $ i \in [m]$, since $\opr\!\vphantom{=}^\oplus = \ior$ in the precedence-only case.
    Now since $I$ is a down-set that contains $b$ and $a_i \ior b$ for all $i\in [m]$, we have $a_i \in I$ again for all such $ i$.
    This allows us to write 
    $$a (\opr|_I) a_1 (\opr|_I) \dots (\opr|_I) a_m (\opr|_I) b,$$
    i.e. $a (\opr|_I)^\oplus b$.
    The second statement now follows immediately from the fact that $\opr\!\vphantom{=}^\oplus = \ior$ in the precedence-only case.
\end{proof}

Since $\oppo_I \,\, = \ior|_I$ for a given down-set $I$, the counterexamples to antisymmetry in $\ior$ are the only possible contenders for counterexamples to antisymmetry in $\oppo_I$, whether or not $I$ belongs to $\PossIm(D)$ is determined by whether or not any of said counterexamples belong to $I$.
As a consequence of this, removing from $\Sigma$ all elements which constitute counterexamples to the antisymmetry of $\ior$, as well as everything above these elements, leaves us with a maximum element $\Imax$ of $\PossIm(D)$ with useful properties that we now prove.

\begin{proposition}\label{prec_possim}
    Let $D = (\Sigma, \Const)$ be a precedence-only declarative process and define $\Imax$ to be the set $\Sigma \backslash \{ a \in \Sigma : b \ior c \ior b \ior a$ for some distinct $b,c \in \Sigma \}$. 
    Then
	\begin{enumerate}
        \item[(i)] $(\Imax, \ior|_\Imax)$ is a poset,
	\item[(i)] $\PossIm(D)$ is the set of down-sets of $(\Imax, \ior|_\Imax)$, and 
	\item[(ii)] $(I, \oppo_I)$ is an induced subposet of $(\Imax, \ior|_\Imax)$ for each $I \in \PossIm(D)$.
	\end{enumerate}
\end{proposition}

\begin{proof}
    (i) This follows immediately from the fact that $\Imax$ contains no counterexamples to antisymmetry, by definition.
    
    (ii) Let $I \in \PossIm(D)$.
    We have that $I$ is a down-set of $(\Sigma, \ior)$ by Proposition~\ref{prec_resp_possim}. 
    Let $a \in I$ and suppose for a contradiction that $a \notin \Imax$. 
    Then there exist distinct $b,c \in \Sigma$ such that $b \ior c \ior b \ior a$.
    Now $b,c \in I$ since $I$ is a down-set of $(\Sigma, \ior)$, so $b (\ior|_I) c (\ior|_I) b$. 
    By Lemma~\ref{prec_lemma}, we have $b \oppo_I c \oppo_I b$.
    Hence $\oppo_I$ is not antisymmetric, contradicting Proposition~\ref{prec_resp_possim}.
    Therefore $I \subseteq \Imax$. 
    Now let $b \in I$ and $a (\ior|_\Imax) b$.
    As $I$ is a down-set of $(\Sigma, \ior)$ we have that $a \in I$.
    Thus $I$ is a down-set of $(\Imax, \ior|_\Imax)$.

    Now let $I$ be a down-set of $(\Imax, \ior|_\Imax)$. 
    Let $x \in I \subseteq \Imax$ and $a \ior x$ and suppose for a contradiction that $a \notin \Imax$.
    Then there exist distinct $b,c \in \Sigma$ such that $b \ior c \ior b \ior a \ior x$, contradicting $x \in \Imax$.
    It follows that $a (\ior|_M) x$.
    Since $I$ is a down-set of $(\Imax, \ior|_\Imax)$, we have that $a \in I$.
    Therefore $I$ is a down-set of $(\Sigma, \ior)$.
    Now suppose for a contradiction that $I \notin \PossIm(D)$.
    By Proposition~\ref{prec_resp_possim}, $\oppo_I$ must not be antisymmetric, so there exist distinct $b,c \in I$ such that $b \oppo_I c \oppo_I b$. 
    By Lemma~\ref{prec_lemma}, we have that $b (\ior|_I) c (\ior|_I) b$. 
    Now we have $b \in \Imax$ and $b \ior c \ior b \ior b$, contradicting the definition of $\Imax$.

    (iii) Let $I \in \PossIm(D)$ and let $a,b \in I$. 
    Then $I \subseteq \Imax$ and 
    $$a \oppo_I b \iff a \ (\opr\!\vphantom{=}^\oplus)|_I \ b \iff a \ (\opr\!\vphantom{=}^\oplus)|_{\Imax} \ b \iff a (\ior|_\Imax) b.$$ 
\end{proof}

\begin{example}\label{example:five:four}
    Consider again the declarative process $D$ from Example~\ref{prec_example}. 
    Let us examine the Hasse diagram of $(\Sigma, \ior)$ in Figure~\ref{fig:prec_hasse} in order to determine $\Imax$.
    In the context of Hasse diagrams, $\Imax := \Sigma \backslash \{ a \in \Sigma ~:~ b \ior c \ior b \ior a$ for some distinct $b,c \in \Sigma \}$ is just $\Sigma$ without those elements which are part of, or above, a directed cycle. 
    In our example, this is readily seen to be the set $\{a,b,c\}$.
    The down-sets of $(\Imax, \ior|_\Imax)$ are $\Emptyset, \{a\}, \{b\}, \{a,b\}$ and $\{a,b,c\}$ so by Proposition~\ref{prec_possim}, these are the sets which constitute $\PossIm(D)$.
\end{example}

Recall that, by Lemma~\ref{maxpossim}, the pair $(\Sigma, \MaxPossIm(D))$ is an independence system, where $\MaxPossIm(D) := \{ \mathsf{max}(I) : I \in \PossIm(D) \}$.
To further justify our interest in this special case, we prove the following specialization of Lemma~\ref{maxpossim}.

\begin{lemma}\label{prec_maxpossim}
    If $D$ is a precedence-only declarative process, then $\MaxPossIm(D) = \mathcal{P}(\mathsf{max}(M))$, where $\mathcal{P}$ denotes the power set.
\end{lemma}

\begin{proof}
Proposition~\ref{prec_possim} states that $\PossIm(D) = \{ I : I \text{ is a down-set of } (\Imax, \ior|_\Imax) \}$.
    Hence $\MaxPossIm(D) = \{ \mathsf{max}(I) : I \text{ is a down-set of } (\Imax, \ior|_\Imax) \}$.
    If $I$ is a down-set of $(\Imax, \ior|_\Imax)$ then $I \subseteq \Imax$ so $\mathsf{max}(I) \subseteq \mathsf{max}(\Imax)$ and it follows that $\MaxPossIm(D) \subseteq \mathcal{P}(\mathsf{max}(M))$. 
    For the converse, notice that $\Imax$ is itself a down-set of $(\Imax, \ior|_\Imax)$ and hence $\mathsf{max}(\Imax)$ is a member of the set $\MaxPossIm(D)$.
    Since $(\Sigma, \MaxPossIm)$ is an independence system by Lemma~\ref{maxpossim}, every subset of $\mathsf{max}(\Imax)$ is also in $\MaxPossIm(D)$.
\end{proof}

In light of both Theorem~\ref{prec_resp_traces} and Proposition~\ref{prec_possim}, it can be seen that the poset $(\Imax, \ior|_\Imax)$ contains all the information we need in order to describe $\Traces(D)$.
This brings us to the classification of $\Traces(D)$.

\begin{theorem}\label{prec_traces}
    Recall $\ior$ from Definition~\ref{ior}. 
    If $D$ is a precedence-only declarative process, then
    $$\Traces(D) = \displaystyle\bigcup_{I \in \mathcal{J}((\Imax, \ior|_\Imax))} \LE((I, \ior|_I)),$$
    where $\Imax := \Sigma \backslash \{ a \in \Sigma ~:~ b \ior c \ior b \ior a$ for some distinct $b,c \in \Sigma \}$.
\end{theorem}

\begin{proof}
    From Theorem~\ref{prec_resp_traces}, we have that
    $$\Traces(D) = \displaystyle\bigcup_{I \in \PossIm(D)} \LE((I, \oppo_I)).$$
    Now by Proposition~\ref{prec_possim} we have that $\PossIm(D) = \mathcal{J}((\Imax, \ior|_\Imax))$, and also that $(I, \oppo_I)$ is an induced subposet of $(\Imax, \ior|_\Imax)$ for each $I \in \PossIm(D)$.
    That is, $(I, \oppo_I) = (I, (\ior|_\Imax)|_I) = (I, \ior|_I)$. 
    The result follows.
\end{proof}

\begin{example}
    Consider again the declarative process $D$ from Example~\ref{prec_example}. 
    We found in Example~\ref{example:five:four} that $\Imax = \{a,b,c\}$ and that the down-sets of $(\Imax, \ior|_\Imax)$ are $\Emptyset, \{a\}, \{b\}, \{a,b\}$ and $\{a,b,c\}$.
    Trivially, we have $\LE((\Emptyset, \ior|_\Emptyset)) = \{\epsilon\}$, $\LE((\{a\}, \ior|_{\{a\}})) = \{a\}$ and $\LE((\{b\}, \ior|_{\{b\}})) = \{b\}$. 
    From Figure~\ref{fig:prec_hasse} we can see that $\LE((\{a,b\}, \ior|_{\{a,b\}})) = \{ab, ba\}$ and $\LE((\{a,b,c\}, \ior|_{\{a,b,c\}})) = \{abc, bac\}$.
    Hence by Theorem~\ref{prec_traces} we have that $\Traces(D) = \{ \epsilon, a, b, ab, ba, abc, bac \}$.
\end{example}

It is well known~\cite[p.295]{stanley} that the linear extensions of a poset $P$ are in bijection with the maximal chains of $\mathcal{J}(P)$. 
These can be seen as paths from $\Emptyset$ to $P$ in the Hasse diagram of $P$. 
Notice that the maximal chains in the down-sets of $\mathcal{J}(P)$ (not to be confused with the elements of $\mathcal{J}(P)$, which are down-sets of $P$) are exactly the saturated chains in $\mathcal{J}(P)$ which include $\Emptyset$.
Theorem~\ref{prec_traces} now reveals a bijection between $\Traces(D)$ and saturated chains in $\mathcal{J}((\Imax, \ior|_\Imax))$ which include $\Emptyset$, or equivalently, directed paths from $\Emptyset$ to any other element in the Hasse diagram of $\mathcal{J}((M, \ior|_\Imax))$.

It is not difficult to see that an implementation of Theorem~\ref{prec_traces} would be simpler than that of the general case, thanks to Lemma~\ref{prec_maxpossim}.
It would only involve the following steps:

\begin{enumerate}
    \item Determine $\Imax$.
    \item Take the downward closure of each $A \subseteq \mathsf{max}(\Imax)$ to obtain $\PossIm(D)$.
    \item For each $I \in \PossIm(D)$, restrict $\ior$ to $I$ and generate the linear extensions of $(I, \ior|_I)$.
\end{enumerate}

It seems that we should be able to determine all of these linear extensions without having to apply a linear extension generating algorithm to each and every poset, 
seeing as how they are all just induced sub-posets of $(\Imax, \ior|_\Imax)$.
This naturally leads us to ask the question; given an induced subposet $Q$ of a poset $P$: can we determine $\LE(Q)$ given $\LE(P)$?

\begin{lemma}\label{induced_subposet_LE}
    Let $P$ be a poset and let $Q$ be an induced subposet of $P$. Then
    $$\LE(Q) ~=~ \{ \pi|_Q ~ : ~ \pi \in \LE(P) \}.$$
\end{lemma}

\begin{proof}
    First notice that $\pi|_Q$, the restriction of the sequence $\pi$ to the set $Q$, is equivalent to the totally ordered set $(Q, \preceq_\pi\!\!|_Q)$, since removing elements from the linear extension $\pi$ does not change the order between the remaining elements.
    Again, by abuse of notation, we let $\pi|_Q$ refer to both the sequence and the totally ordered set.

    Let $\sigma \in \LE(Q)$.
    It can be shown that $(P, (\preceq_P \cup \preceq_\sigma)^\oplus)$ is a poset, which we denote by $P + \sigma$. 
    Since $(\preceq_P) \subseteq (\preceq_P \cup \preceq_\sigma)^\oplus$, we have that $\LE(P + \sigma) \subseteq \LE(P)$.
    Now if $\pi \in \LE(P + \sigma)$ then $\pi \in \LE(P)$ and clearly $\pi|_Q = \sigma$.
    
    Conversely, let $\pi \in \LE(P)$.
    Let $a,b \in Q$ such that $a \preceq_Q b$.
    Then $a \preceq_P b$ since $Q$ is an induced subposet of $P$.
    So $a \preceq_\pi b$ since $\pi$ is a linear extension of $P$.
    Finally we have that $a (\preceq_\pi\!\!|_Q) b$ since $a,b \in Q$.
    This shows that $\pi|_Q \in \LE(Q)$.
\end{proof}

Now we need only apply a linear extension generating algorithm to the poset $(\Imax, \ior|_\Imax)$, the remaining $\LE((I, \oppo_I))$ are then easily found using Lemma~\ref{induced_subposet_LE}. 

\begin{example}
    Consider again the declarative process $D$ from Example~\ref{prec_example}.
    We have found previously that $\Imax = \{a,b,c\}$, that the down-sets of $(\Imax, \ior|_\Imax)$ are $\Emptyset, \{a\}, \{b\}, \{a,b\}$ and $\{a,b,c\}$, and that $\LE((\Imax, \ior|_\Imax)) = \{ abc, bac \}$. 
    Consider Theorem~\ref{prec_traces} in the context of our example. 
    Every $(I, \ior|_I)$ is an induced subposet of $(\Imax, \ior|_\Imax)$ since $I \subseteq M$, and so by repeated application of Lemma~\ref{induced_subposet_LE}, we find that 
    \begin{align*}
        \Traces(D) = & \left\{ abc|_\Emptyset, bac|_\Emptyset, abc|_{\{a\}}, bac|_{\{a\}}, abc|_{\{b\}}, bac|_{\{b\}}, abc|_{\{a,b\}}, bac|_{\{a,b\}},  abc|_{\{a,b,c\}}, bac|_{\{a,b,c\}}\right\}\\
                  =  & \{ \epsilon, a, b, ab, ba, abc, bac\}.
    \end{align*}
\end{example}

Although it is true that in practice, the application of Lemma~\ref{induced_subposet_LE} will usually produce a large amount of duplicate sequences, it offers an alternative which only involves one application of a linear extension generating algorithm, should this be preferable.

\subsection{Response-only constraint sets}
The response-only case is almost identical to the precedence-only case. 
In this case, $\ior$ and $\opr\!\vphantom{=}^\oplus$ are each others transpose, rather than being equal.
From here all the results of the previous section hold, save for Lemma~\ref{prec_lemma}, where $\oppo_I$ and $\ior|_I$ are each others transpose rather than equal.

\subsection{Successor-only constraint sets}
Consider a declarative process $D = (\Sigma, \Const)$ where $\Const$ consists only of successor constraints. 
We refer to such a declarative process as a \textit{successor-only} declarative process.
Recall that $\Succ(a,b)$ can be rewritten as $\Resp(a,b) \wedge \Prec(a,b)$, so that this is in fact a special case of the precedence-and-response-only declarative process.
In this case, it follows from Definition~\ref{ior} that $\ior$ is the relation $\{(a,b) \in \Sigma^2 : \Succ(a,b) \in \Const \text{ or } \Succ(b,a) \in \Const \}^\oplus$ on $\Sigma$ and from Definition~\ref{opr} that $\opr$ is the relation $\{(a,b) \in \Sigma^2 : \Succ(a,b) \in \Const \}$ on $\Sigma$.

Notice that $\ior$ is an equivalence relation on $\Sigma$.
We will show that the $\ior$-equivalence classes can be thought of as the basic building blocks of the declarative process, and we will use them to classify $\Traces(D)$.

\begin{lemma}\label{succ_lemma}
    Let $D = (\Sigma, \Const)$ be a successor only declarative process.
    Then $I$ is a down-set of $(\Sigma, \ior)$ if and only if $I$ is a union of elements of $\Sigma / \ior$.
    Furthermore, if $K_1, \dots, K_m \in \Sigma / \ior$ (so that $K_1 \cupdot \dots \cupdot K_m$ is a down-set of $(\Sigma, \ior)$), then $\oppo_{K_1 \cupdot \dots \cupdot K_m} \, = \, \oppo_{K_1} \cupdot \dots \cupdot \oppo_{K_m}$.
\end{lemma}

\begin{proof}
    Let $I$ be down-set of $(\Sigma, \ior)$.
    If $a \in I$ and $b \in [a]_{\ior}$ then $b \ior a$ and hence $b \in I$ since $I$ is a down-set of $(\Sigma, \ior)$.
    So $[a]_{\ior} \subseteq I$, which implies that $I$ is a union of elements of $\Sigma / \ior$.

    Let $I$ be a union of elements of $\Sigma / \ior$. 
    Let $a \in I$ and $b \ior a$.
    Then $b \in [a]_{\ior} \subseteq I$ so $I$ is a down-set of $(\Sigma, \ior)$.

    For the second statement, notice that $\ior = \opr \cup \opr\!\vphantom{=}^T$ in the successor-only case, meaning that the $\ior$-equivalence classes also partition $\Sigma$ into its $\opr$-weakly connected components. 
    This implies the following identity given $K_1, \dots, K_m \in \Sigma / \ior$.
    \begin{align*}
        (\oppo_{K_1 \cupdot \dots \cupdot K_m}) 
		&:= (\opr|_{K_1 \cupdot \dots \cupdot K_m})^\oplus \\ 
		& = ((\opr|_{K_1}) \cupdot \dots \cupdot (\opr|_{K_m}))^\oplus \\ & = (\opr|_{K_1})^\oplus \cupdot \dots \cupdot (\opr|_{K_m})^\oplus\\ 
        &=: \, \oppo_{K_1} \cupdot \dots \cupdot \oppo_{K_m}. \qedhere
    \end{align*}
\end{proof}

It follows from the second part of this proposition that, given a down-set $I$ of $(\Sigma, \ior)$, its corresponding order relation $\oppo_I$ is antisymmetric if and only if the order relation corresponding to each of its constituent $\ior$-equivalence classes is antisymmetric. 
But recall that we are only interested in those down-sets whose corresponding order relations are antisymmetric, and hence, those equivalence classes whose corresponding order relations are not antisymmetric are useless to as building blocks, and so we discard them to form the set $\CoreOcc(D)$.

\begin{proposition}\label{succ_possim}
    Let $D = (\Sigma, \Const)$ be a successor-only declarative process, and define $\CoreOcc(D)$ to be the set $\{ K \in \Sigma / \ior : \,\, \oppo_{K} \text{ is antisymmetric} \}$.
    Then $\PossIm(D)$ is the set of all possible unions of the elements of $\CoreOcc(D)$.
\end{proposition}

\begin{proof}
    Let $I \in \PossIm(D)$. 
    Then $I$ is a down-set of $(\Sigma, \ior)$ and $\oppo_I$ is antisymmetric by Proposition~\ref{prec_resp_possim}.
    Hence $I = K_1 \cupdot \dots \cupdot K_m$ for some $K_1, \dots, K_m \in \Sigma / \ior$ and $\oppo_I \,\, = \, \oppo_{K_1} \cupdot \dots \cupdot \oppo_{K_m}$ by Lemma~\ref{succ_lemma}. 
    Now $\oppo_{K_1}, \dots, \oppo_{K_m}$ must all be antisymmetric since $\oppo_I$ is antisymmetric. 
    Hence $K_1, \dots, K_m \in \CoreOcc(D)$.

    Let $I = K_1 \cupdot \dots \cupdot K_m$ for some $K_1, \dots, K_m \in \CoreOcc(D)\subseteq \Sigma / \ior$. 
    Then $I$ is a down-set of $(\Sigma, \ior)$ and $\oppo_I \,\, = \, \oppo_{K_1} \cupdot \dots \cupdot \oppo_{K_m}$ by Lemma~\ref{succ_lemma}.
    Now $\oppo_I$ must be antisymmetric since $\oppo_{K_1}, \dots, \oppo_{K_m}$ are all antisymmetric. 
    Hence $I \in \PossIm(D)$ by Proposition~\ref{prec_resp_possim}.
\end{proof}

This brings us to our classification of $\Traces(D)$.

\begin{theorem}\label{succ_traces}
    Recall that in the successor only case, $\ior$ and $\opr$ are the relations 
	\begin{align*}
	\ior & = \{(a,b) \in \Sigma^2 : \Succ(a,b) \in \Const \text{ or } \Succ(b,a) \in \Const \}^\oplus, \\
	\opr &= \{(a,b) \in \Sigma^2 : \Succ(a,b) \in \Const \}.
	\end{align*}
    Also recall $\oppo_I$ from Definition~\ref{oppo}.
    If $D$ is a successor-only declarative process, then
    $$\Traces(D) = \displaystyle\bigcup_{S \subseteq \CoreOcc(D)} \LE \left( \left( \displaystyle\bigcupdot_{K \in S} K, \, \displaystyle\bigcupdot_{K \in S} \oppo_K \right) \right),$$
    where $\CoreOcc(D) = \{ K \in \Sigma / \ior  ~:~ \oppo_K \text{ is antisymmetric} \}$.
\end{theorem}

\begin{proof}
    From Theorem~\ref{prec_resp_traces}, we have that
    $$\Traces(D) = \displaystyle\bigcup_{I \in \PossIm(D)} \LE((I, \oppo_I)).$$
    Proposition~\ref{succ_possim} can be rewritten as $\PossIm(D) = \{ \bigcupdot_{K \in S} K : S \subseteq \CoreOcc(D) \}$ and Lemma~\ref{succ_lemma} implies that if $I = \bigcupdot_{K \in S} K \in \PossIm(D)$, then $\oppo_I = \bigcupdot_{K \in S} \oppo_K$. 
    The result follows.
\end{proof}

As with the precedence-only case, an implementation of this theorem would be simpler than that of the general case.
It would only involve the following steps:
\begin{enumerate}
    \item generate $\Sigma / \ior$,
    \item generate $\oppo_K$ for each $K \in \Sigma / \ior$ and 
    \item check which $\oppo_K$ are antisymmetric in order to find $\CoreOcc(D)$, and
    \item for each $S \subseteq  \CoreOcc(D) $, generate the linear extensions of $(\bigcupdot_{K \in S} K, \, \bigcupdot_{K \in S} \oppo_K)$.
\end{enumerate}

\section*{Acknowledgments}
\noindent M.D. wishes to acknowledge the financial support from University College Dublin (OBRSS Research Support Scheme 16426) for this work.

\end{document}